\documentclass[reqno]{amsart}
\usepackage{amsmath,amssymb,latexsym}
\usepackage{mathpazo}
\usepackage[mathscr]{eucal}
\usepackage{enumerate}
\usepackage[all]{xy}

\newcommand{\calW}{{\mathcal{W}}}

\newcommand{\calC}{\mathcal{C}}
\newcommand{\calF}{\mathcal{F}}

\newcommand{\calE}{\mathcal{E}}
\newcommand{\calJ}{\mathcal{J}}

\newcommand{\sfJ}{\mathsf{J}}

\newcommand{\frq}{\mathfrak{q}}

\newcommand{\frakm}{\mathfrak{m}}

\newcommand{\C}{\mathbb{C}}
\newcommand{\R}{\mathbb{R}}

\newcommand{\N}{\mathbb{N}}
\newcommand{\Z}{\mathbb{Z}}

\newcommand{\Tot}{\operatorname{Tot}}

\newcommand{\id}{\operatorname{id}}

\newcommand{\fWeyl}{\mathbb W}

\newcommand{\Sym}{\operatorname{Sym}}
\newcommand{\cSym}{\operatorname{\widehat{Sym}}}

\newcommand{\supp}{\operatorname{supp}}

\numberwithin{equation}{section}
\theoremstyle{plain}
        \newtheorem{theorem}{Theorem}[section]
        
        \newtheorem{proposition}[theorem]{Proposition}

\theoremstyle{definition}
        \newtheorem{definition}[theorem]{Definition}
        \newtheorem{remark}[theorem]{Remark}

\title{Quantization of Whitney functions}
\author{M.J.~Pflaum, H. Posthuma,~\textrm{and} X.~Tang}
\begin{document}
\begin{abstract}
  We propose to study deformation quantizations of Whitney functions. 
  To this end, we extend the notion of a deformation quantization to algebras of Whitney 
  functions over a singular set, and show the existence of a deformation quantization 
  of Whitney functions over a closed subset of a symplectic manifold. 
  Under the assumption that the underlying symplectic manifold is analytic and the 
  singular subset subanalytic, we determine that the Hochschild and cyclic homology of the 
  deformed algebra of Whitney functions over the subanalytic subset coincide with 
  the Whitney--de Rham cohomology.  
  Finally, we note how an algebraic index theorem for Whitney functions can be derived. 
\end{abstract}
\dedicatory{Dedicated to the memory of our friend and collaborator 
            Nikolai Neumaier}
\address{\newline
Markus J. Pflaum, {\tt markus.pflaum@colorado.edu}\newline
         \indent {\rm Department of Mathematics, University of Colorado,
         Boulder, USA}\newline
        Hessel Posthuma, {\tt h.b.posthuma@uva.nl}\newline
         \indent {\rm Korteweg-de Vries Institute for Mathematics, University
         of Amsterdam, The Netherlands} \newline
        Xiang Tang, {\tt xtang@math.wustl.edu}   \newline
         \indent {\rm  Department of Mathematics, Washington University,
         St.~Louis, USA}}
\maketitle
\section*{Introduction}
In physics, many interesting systems are described mathematically by 
phase spaces with singularities such as for example  the moduli spaces of 
flat connections on a Riemann surface. The study of such singular phase spaces
raises a very interesting question in mathematical physics. How does one
quantize a singular Poisson manifold? In his seminal paper \cite{KonDQPM}, 
Kontsevich completely solved the problem of constructing 
deformation quantizations of Poisson manifolds by his famous formality theorem.
However, the problem of proving a general existence theorem for 
deformation quantizations over singular spaces is still open 15 years 
later (see \cite{BorHerPflHASRDQ,HerPflIyeESPQSLHTA} for progress in this direction).

One of the key difficulties in the quantization theory of 
singular phase spaces spaces is that the algebra of smooth functions 
over  a space with singularities appears to be complicated to study 
since certain crucial results such as a de Rham Theorem or a 
Hochschild--Kostant--Rosenberg type theorem do in general not hold true
in the presence of singularities. 

In this paper, we propose to replace the algebra of smooth functions by the 
so-called Whitney functions, and discuss some examples of quantizations 
of Whitney functions.  

Let $M$ be a smooth manifold, and $X\subset M$ be a closed subset of $M$. 
A Whitney function on $X$, roughly speaking, is the (infinite) jet of a 
smooth function $f$ on $M$ at the subset $X$. We denote the algebra of 
Whitney functions on $X$ by $\calE^\infty(X)$. A Whitney--Poisson structure on 
$X$ is a Poisson structure on $\calE^\infty(X)$, i.e. an antisymmetric bilinear 
bracket $\{ - , - \}$ on $\calE^\infty(X)$ which is a derivation in each of its arguments 
and satisfies the Jacobi-identity. 
Several  interesting questions arise in the study of Whitney--Poisson structures. 
\begin{enumerate}
\item 
First observe that if a neighborhood of $X$ in $M$ is equipped with a Poisson bivector $\Pi$, 
then $\Pi$ naturally defines a Whitney--Poisson structure on  $X$. This construction usually 
provides 
various different Whitney--Poisson structures on $X$, which we will call global Whitney--Poisson structures. 
In general, is every Whitney--Poisson structure on $X$ a global one? This question is closely related to the 
existence of a normal form of a Poisson structure near $X$. We expect to see obstructions for a general $X$ 
in $M$, which is probably connected to the singularities of $X$ and the embedding of $X$ in $M$. 
\item 
Whitney functions naturally factorize to smooth functions on $X$. In general, a Whitney--Poisson structure 
does not factorize to  a Poisson structure on $X$ by which we mean an antisymmetric and bilinear 
bracket on $\calC^\infty (X)$ which is a derivation in each of its arguments and satisfies the 
Jacobi-identity.
It appears to be an interesting question to describe those Whitney--Poisson 
structures that do factorize to $X$. 
This problem appears to be  closely related to the question under which conditions one 
can embed a singular Poisson variety into a smooth Poisson manifold, see 
\cite{EgiEBD,DavisEBD,McMillanEBD}. 
\end{enumerate}
In this paper, we propose to study the problem of deformation quantization of Whitney--Poisson 
structures on $X$. We will construct a natural deformation quantization of a global 
Whitney--Poisson structure on $X$. Moreover, we study such a deformation quantization by 
computing its  Hochschild homology when the global Whitney--Poisson structure is symplectic 
using the methods developed in \cite{PPT2ADV}.

We would like to dedicate this short article to Nicolai Neumaier, who unfortunately passed away  
in Spring 2010 after a brave and long battle with cancer.  Nicolai has been a good friend and 
excellent collaborator. The idea to study the quantization of Whitney functions goes back to our 
collaboration in 2004 on deformation quantization of orbifolds \cite{nppt}.  
We are picking up this idea as a memory to Nicolai's important contribution to the subject of 
deformation quantization of singular spaces. 

\noindent{\bf Acknowledgments:} Pflaum is partially supported by NSF grant DMS 1105670, and 
Tang is partially supported by NSF grant DMS 0900985.


%
%
\section{Formal quantizations of  Whitney functions}
Assume to be given a smooth manifold $M$, and let $X\subset M$ be a  closed subset.
Denote by $\calJ^\infty (X,M) \subset \calC^\infty (M)$ the ideal of smooth functions on $M$ which are 
flat  on $X$, i.e.~the space of all $f \in \calC^\infty (M)$ such that for every differential 
operator $D$ on $M$ the restricted function $Df_{|X}$ vanishes. 
By Whitney's Extension Theorem, the quotient 
$\calE^\infty (X) := \calC^\infty (M) / \calJ^\infty (X,M)$ naturally coincides 
with the algebra of Whitney functions on $X$. This implies in particular that 
  $\calE^\infty (X) \subset \sfJ^\infty (X)$, where $\sfJ^\infty (X)$ denotes the space of 
infinite jets over $X$.
Now consider the complex $\Omega (M)$ of differential 
forms on $M$. Then the spaces $ \Omega^k_{\calJ^\infty}  (X,M) :=   \calJ^\infty (X,M) \cdot \Omega^k (M)$
are modules over $\calC^\infty (M)$ preserved by the exterior derivative $d$, which means that 
$d \big( \Omega^k_{\calJ^\infty}  (X,M) \big) \subset \Omega^{k+1}_{\calJ^\infty}  (X,M)$. 
One thus obtains a subcomplex $\Omega^\bullet_{\calJ^\infty}  (X,M) \subset \Omega^\bullet (M)$ 
which we call the complex of differential forms on $M$ which are flat on $X$. 
The quotient complex 
$ \Omega^\bullet_{\calE^\infty} (X) := \Omega^\bullet (M)/ \Omega^\bullet_{\calJ^\infty}  (X,M)$ will be called 
the complex of Whitney-de Rham forms on $X$. According to \cite{BraPfl}, the cohomology of 
$\Omega^\bullet_{\calE^\infty} (X)$ coincides with the singular cohomology (with values in $\R$), if $M$ 
is an analytic manifold, and $X\subset M$ a subanalytic subset.  

Let us now define what we understand by a deformation quantization of Whitney functions.
\begin{definition}
  Assume to be given a manifold $M$, a closed subset $X \subset M$ and a \emph{Whitney--Poisson} structure  
  on $X$, i.e.~a bilinear map $\{ - , - \}$ on $\calE^\infty (X)$ which satisfies  for all $F,G,H \in \calE^\infty (X)$
  the relations
 \begin{enumerate}[(P1)]
  \item $\{F, GH\}=\{F,G\}H+G\{F,H\}$, and 
  \item $\{\{F,G\}, H\}+\{\{H,F\},G\}+\{\{G,H\},F\}=0$.
\end{enumerate}
 
  By a \emph{formal deformation quantization} of the algebra $\calE^\infty (X)$  or in other words 
  a \emph{star product} on $\calE^\infty (X)$ we understand an associative product 
  \[ 
   \star : \calE^\infty (X)[[\hbar]] \times \calE^\infty (X)[[\hbar]] \rightarrow \calE^\infty (X)[[\hbar]] 
  \]
  on the space $\calE^\infty (X)[[\hbar]]$ of formal power series in the variable $\hbar$ with coefficients 
  in $\calE^\infty (X)$ such that the following is satisfied:
  \begin{enumerate}[(DQ1)]
  \setcounter{enumi}{-1}
  \item
      The product $\star $ is $\R[[\hbar]]$-linear and $\hbar$-adically continuous in each argument.
  \item
     There exist $\R$-bilinear operators $c_k : \calE^\infty (X) \times \calE^\infty (X) \rightarrow \calE^\infty (X)$, $k\in \N$ 
     such that $c_0 $ is the standard commutative product on $\calE^\infty (X)$ and such that for all $F,G \in \calE^\infty (X)$ 
     there is an expansion of the product $ F \star G $ of the form
     \begin{equation}
     \label{eq:starprodexp}
       F \star G = \sum_{k\in \N} c_k (F,G) \hbar^k .
     \end{equation}
  \item 
     The constant function $1 \in \calE^\infty$ satisfies $1 \star F = F \star 1 = F$  for all $F \in \calE^\infty (X)$.
  \item
     The star commutator $[ F, G]_\star := F\star G - G \star F$ of two Whitney functions $F,G \in \calE^\infty (X)$
     satisfies the  commutation relation 
     \[
        [ F, G]_\star = - i \hbar \{ F, G\} + o(\hbar^2) .
     \]
  \end{enumerate}
    If in addition $\star$ is local in the sense that 
   \begin{enumerate}[(DQ1)]
   \setcounter{enumi}{3}
  \item 
    $\supp (F \star G) \subset \supp (F) \cap \supp (G)$ for all $F,G \in \calE^\infty (X)$,
  \end{enumerate}
    then the star product is called \emph{local}. 
\end{definition}
\begin{remark}
  If $(M,\Pi)$ is a Poisson manifold, the ideal $\calJ^\infty (X;M)$ is a even Poisson 
  ideal in $\calC^\infty (M)$. 
  This implies that the Poisson bracket on $\calC^\infty (M)$ factors to the 
  quotient $\calE^\infty (X)$. 
  We denote the inherited Poisson bracket on $\calE^\infty (X)$ also by $\{ - , - \}$, and call it 
  \emph{global Whitney--Poisson structure}.  
\end{remark}

Assume now to be given a Poisson manifold $(M,\Pi)$, a  closed subset $X\subset M$, and let $\star$ be a local 
star product on $\calC^\infty (M)$. By Peetre's Theorem one then knows that each of the operators 
$c_k: \calC^\infty (M) \times  \calC^\infty (M) \rightarrow \calC^\infty (M)$ in the expansion 
Eq.~\eqref{eq:starprodexp} of the star product on  $\calC^\infty (M)$ is locally bidifferential. But this implies that 
for every $k\in \N$ the sets 
$c_k \big( \calJ^\infty (X,M ) \times \calC^\infty (M) \big) $ and $c_k \big( \calC^\infty (M) \times \calJ^\infty (X,M ) \big) $
are contained in $\calJ^\infty (X,M )$. This immediately entails the following result.

\begin{proposition}
  Let $(M,\Pi)$ be a Poisson manifold and $\star$ a local star product on $\calC^\infty (M)$. Then 
  for each closed subset $X \subset M$ the subspace  $\calJ^\infty (X,M )[[\hbar]]$ is an ideal in 
  $\big( \calC^\infty (M) , \star \big)$ which gives rise to an exact sequence of deformed algebras
  \[
    0 \rightarrow \big( \calJ^\infty (X,M )[[\hbar]] , \star \big)
      \rightarrow \big( \calC^\infty (M) , \star \big)  \rightarrow \big( \calE^\infty (X) , \star \big) 
      \rightarrow 0 ,
  \]
  where the induced star product on $\calE^\infty (X)$ is denoted by $\star$ as well. 
\end{proposition}

\begin{remark}
  One knows by the work of {\sc Fedosov} \cite{FedDQIT} that every symplectic manifold carries a local star product, and 
  by {\sc Kontsevich} \cite{KonDQPM} that on every Poisson manifold there exists  a local star product. 
  The proceeding proposition then entails that for every closed subset $X$ of a Poisson manifold  $(M,\Pi)$
  there exists a deformation quantization of $\calE^\infty (X)$ with the induced global Whitney--Poisson structure.
\end{remark}

Let us briefly recall Fedosov's approach \cite{FedDQIT} for the construction of  a deformation quantization over a symplectic 
manifold $(M,\omega)$ and use this to describe the induced star product on $\calE^\infty (X)$ with $X\subset M$ closed
in more detail. 
To this end, observe first that each of the tangent spaces $T_pM$ is a linear symplectic space,
hence  gives rise to the formal Weyl algebra $\fWeyl  (T_p M ) $. As a vector space, $\fWeyl  (T_p M ) $
coincides with $\cSym (T^*_pM)[[\hbar]]$, the space of formal power series in $\hbar$ with coefficients
in the space of  Taylor expansions at the origin of smooth functions on $T_pM$. Note that  $\cSym (T^*_pM)$
coincides with the $\frakm$-adic completion of the space $\Sym (T^*_pM)$ of polynomial functions on $T_pM$,
where $\frakm $ denotes the maximal ideal in $\Sym (T^*_pM)$. In other words this means that $\cSym (T^*_pM)$
can be identified with the product $\prod_{s\in \N} \Sym^s (T^*_pM)$, where $\Sym^s (T^*_pM)$ denotes the space 
of $s$-homogenous polynomial functions on $T_pM$. Hence every element $a$ of $\fWeyl  (T_p M ) $ can be uniquely 
expressed in the form
\begin{equation}
\label{eq:fWeylrep}
  a = \sum_{s\in \N , \, k\in \N} a_{s,k} \hbar^k ,
\end{equation}
where the $a_{s,k} \in \Sym^s (T^*_pM)$ are uniquely defined by $a$. For later purposes note that 
$\fWeyl  (T_p M ) $ is filtered by the \emph{Fedosov-degree}
\[
  \deg_\textup{F} (a) := \min\{s +2k \mid a_{sk}\neq 0 \}, \quad 
  a \in \fWeyl  (T_p M ). 
\]
Next observe that the Poisson bivector $\Pi$ on $T_pM$ is linear and can be 
written in the form  
\begin{equation}
\label{eq:Poissonrep}
	  \Pi = \sum_{i=1}^{\frac{\dim T_pM}{2}} \Pi_{i1} \otimes \Pi_{i2} \quad \text{with $\Pi_{i1},\Pi_{i2} \in T_pM$, 
            $i=1,\cdots , \frac{\dim T_pM}{2}$}.  
\end{equation}
Since the elements of $T_pM$ act as derivations on $\Sym (T_pM)$ one obtains an operator
\begin{equation}
\begin{split}
  \widehat{\Pi} : \: & \Sym (T_pM) \otimes \Sym (T_pM) \rightarrow \Sym (T_pM) \otimes \Sym (T_pM), \\ 
  & a \otimes b \mapsto \sum_{i=1}^{\frac{\dim T_pM}{2}} \Pi_{i1} a \otimes \Pi_{i2} b ,
\end{split}
\end{equation}
which does not depend on the particular representation \eqref{eq:Poissonrep}. Note that by $\C[[\hbar]]$-linearity
and $\frakm$-adic continuity, $\widehat{\Pi}$ uniquely extends to an operator
\[
 \widehat{\Pi} :  \cSym (T_pM)[[\hbar]] \otimes \cSym (T_pM)[[\hbar]] 
  \rightarrow  \cSym (T_pM)[[\hbar]] \otimes \cSym (T_pM)[[\hbar]] .
\] 
The product of two elements $a, b \in \fWeyl  (T_p M ) $ can now be written 
down. It is the so-called \emph{Moyal--Weyl} product of $a$ and $b$ and is 
given  by
\begin{equation}
  a \circ b := \sum \frac{ (- i \hbar)^k}{k!} \mu \big( \widehat{\Pi} 
  ( a \otimes b ) \big) .
\end{equation}
One checks easily that $\circ$ is a star product on $\fWeyl  (T_p M )$. 

Denote by $\fWeyl (M) $ the bundle of formal Weyl algebras over $M$, which is 
the  (profinite dimensional) vector bundle over $M$ having fibers 
$ \fWeyl  (T_p M ) $, $p\in M$. 
Furthermore, let $\Omega^\bullet \fWeyl  $
be the sheaf of smooth differential forms with  values in the bundle $\fWeyl (M)$.
Note that both the space $\calW (M)$ of smooth sections of  $\fWeyl (M) $
and the space $\Omega^\bullet \fWeyl (M)$ are filtered by the Fedosov-degree. More precisely, 
the Fedosov filtration $\big( \calF^k \calW (M) \big)_{k\in \N}$ of $\calW (M)$ is given by
\[
   \calF^k \calW (M) := 
   \{ a \in \calW (M) \mid \deg_\textup{F} (a(p)) \geq k \text{ for all $p \in M$}  \} ,
\]
and similarly for $\Omega^\bullet \fWeyl (M)$. 
Note also that an element $a \in \calW (M)$ can be uniquely written
in the form \eqref{eq:fWeylrep}, where the $a_{s,k}$ with $s,k\in\N$ then are 
smooth sections of the symmetric powers $\Sym^s (T^*M)$. This representation 
allows us to define the \emph{symbol map} 
$\sigma: \calW \rightarrow \calC^\infty (M)[[\hbar]]$ by 
\[
  \sigma (a ) = \sum_{k\in \N} a_{0,k} \hbar^k \quad \text{for $a\in \calW$}.
\]

Next, choose a a symplectic connection $\nabla$ on $M$, i.e.~a connection on $M$ which satisfies 
$\nabla \omega = 0$. The symplectic connection canonically lifts to a connection 
\[
  \nabla : \Omega^\bullet \fWeyl (M ) \rightarrow \Omega^{\bullet + 1} \fWeyl ( M ).
\]
By Fedosov's construction, there exists a section $A \in \Omega^1 \fWeyl (M)$ 
such that the connection
\begin{equation}
  \label{eq:FedConn}
   D : = \nabla + \frac{i}{\hbar} [ - , A ]
\end{equation}
is abelian, i.e.~satisfies $D \circ D = 0$. The $1$-form $A$ is even uniquely 
determined by the latter property, if one additionally requires that 
$\deg_\textup{F} (A) \geq 2$. The connection $D$ defined by such a  $1$-form $A$
will be called a \emph{Fedosov connection}.

As has been observed by Fedosov \cite{FedDQIT}, the space 
\[
 \calW_D (M) := \{ a \in \calW (M) \mid Da = 0 \}
\]
of flat sections of the Weyl algebras bundle gives rise to  a deformation quantization of $\calC^\infty (M)$ via
 the symbol map
\[
  \sigma :  \calW (M) \rightarrow \calC^\infty (M)[[\hbar]],  \quad a = \sum_{s\in \N, k \in \N} a_{s,k} \hbar^k \mapsto
   \sum_{k \in \N} a_{0,k} \hbar^k .
\]
More precisely, if the $1$-form $A$ has been chosen as above, the restriction
\[
  \sigma_{|\calW_D (M)} :  \calW_D (M) \rightarrow \calC^\infty (M)[[\hbar]]
\]
is a linear isomorphism. Let 
\[
  \frq :   \calC^\infty (M)[[\hbar]]  \rightarrow \calW_D (M)
\]
be its inverse, the so-called \emph{quantization map}. 
Then there exist uniquely determined differential operators 
$\frq_k :  \calC^\infty (M) \rightarrow  \calC^\infty (M)$ such that 
\begin{equation}
\label{eq:expquantmap}
   \frq (f) = \sum_{k\in \N} \frq_k (f) \hbar^k \quad 
   \text{for all $f \in \calC^\infty (M)$},
\end{equation}
and 
\[
  \star : \calC^\infty (M)[[\hbar]]  \times \calC^\infty (M)[[\hbar]] , \quad
  (f , g) \mapsto \sigma \big(  \frq (f) \circ \frq (g) \big)
\] 
is a star product on $\calC^\infty (M)$.

Now observe that the Fedosov connection $D$ leaves the module 
$\calJ^\infty (X;M) \cdot \Omega^\bullet (M;\fWeyl M)$ invariant.   
This implies that $D$ factors to the quotient 
\[ 
  \Omega^\bullet_{\calE^\infty} (X;\fWeyl M) := 
  \Omega^\bullet (M;\fWeyl M) / \calJ^\infty (X;M) \cdot \Omega^\bullet (M;\fWeyl M),
\]
and acts on 
$\calE^\infty (X;\fWeyl M) := \calW ( M) / \calJ^\infty (X;M) \cdot \calW  (M)$. 
Moreover, the symbol map $\sigma$ maps  
$\calJ^\infty (X;M)\cdot \calW (M)$ to $\calJ^\infty (X;M)[[\hbar]]$, and 
$\frq \big( \calJ^\infty (X;M)[[\hbar]]\big)$ is contained in $\calJ^\infty (X;M)\cdot \calW (M)$,
since in the expansion \eqref{eq:expquantmap} the operators $\frq_k$ are all differential operators.  
Hence $\sigma$ and $\frq$ factor to $\calE^\infty (X;\fWeyl M)$ respectively 
$\calE^\infty (X)[[\hbar]]$. This entails the following result.

\begin{theorem}
  Let $(M,\omega)$ be a symplectic manifold, $D$ a Fedosov connection on $\Omega^\bullet \fWeyl $,
  and $X\subset M$ a closed subset. Then the space of flat sections 
  \[
    \calW_D (X) := \{ a \in \calE^\infty (X;\fWeyl M)  \mid Da = 0\}
  \] 
  is a subalgebra of $\calE^\infty (X;\fWeyl M)$, and the symbol map induces an isomorphism of linear 
  spaces 
  $\sigma_X : \calW_D (X) \rightarrow \calE^\infty(X)[[\hbar]]$. Moreover, the unique product $\star$ on 
  $\calE^\infty(X)[[\hbar]]$ with respect to which $\sigma_X$ becomes an isomorphism of algebras 
  is  a formal deformation quantization of $\calE^\infty(X)$.
\end{theorem}

\section{Hochschild and cyclic homology}
The Hochschild homology of algebras of Whitney functions $\calE^\infty (X)$ has been computed for a large 
class of singular subspaces $X\subset M$ in \cite{BraPfl}. In particular, it follows from this work that 
for (locally) subanalytic sets $X\subset M$ with $M$ an analytic manifold the Hochschild homology of 
$\calE^\infty (X)$ is given by
\begin{equation}
  \label{eq:HHWhitneyFunc}
  HH_\bullet \big( \calE^\infty (X) \big) =  \Omega^\bullet_{\calE^\infty} (X) .
\end{equation}

In case $(M , \omega)$ is symplectic of dimension $2m$, and $\star$  
a star product on $\calC^\infty (M)$, the 
Hochschild homology  of the deformed algebra $\big( \calC^\infty (M)((\hbar)),\star \big)$ was first computed in \cite{NesTsyAIT}. (We extend the star product $\star$ on $\calC^\infty(M)[[\hbar]]$ to $\calC^\infty(M)((\hbar))$). It is given by  
\begin{equation}
  \label{eq:HHDefAlgSmooth}
  HH_\bullet \big(  \calC^\infty (M)((\hbar))\big) = H^{2m -\bullet}_\textup{dR} (M, \C((\hbar)) ). 
\end{equation}
If $X \subset M$ now is closed, the natural question  arises what the Hochschild homology  of the 
deformed algebra of Whitney functions  $\big( \calE^\infty (X)((\hbar)), \star \big)$ then is.  
Observe that via Teleman's localization technique \cite{Tlocal}, the Hochschild and cyclic homology of $\calE^\infty(X)$ and $\calE^\infty(X)((\star))$ (and also of $\calC^\infty(M)$ and $\calC^\infty(M)((\hbar))$) can be computed as the sheaf cohomology of the corresponding sheaf complexes for 
Hochschild and cyclic complexes on $X$ (and on $M$) as is explained in \cite{BraPfl}.  

We start the computation of the homology groups by first noting that $\calE^\infty (X)((\hbar))$ carries a filtration 
$\big( \calF^k_\hbar \calE^\infty (X)((\hbar))\big)_{k\in \Z}$ by the $\hbar$-degree. More precisely,
\[
   \calF^k_\hbar \calE^\infty (X)((\hbar)) = 
   \{ F \in \calE^\infty (X)((\hbar)) \mid \deg_\hbar F \geq k  \},
\]
where the $\hbar$-degree of $F = \sum_{k\in\Z} F_k \hbar^k \in  \calE^\infty (X)((\hbar))$ with
$F_k  \in  \calE^\infty (X)$ is given by 
\[
  \deg_\hbar (F) = \min \{ k\in \Z \mid F_k \neq 0 \}.
\]
The $\hbar$-filtration of $\calE^\infty (X)((\hbar))$ induces  a filtration 
$\big( \calF^k_\hbar C^\bullet (\calE^\infty (X)((\hbar)) \big)_{k\in \Z}$ of the 
Hochschild chain complex $C^\bullet (\calE^\infty (X)((\hbar)))$ which then gives rise to a
spectral sequence $E^\bullet_{pq}$. Since 
\[
  \calF^{k+1}_\hbar \calE^\infty (X)((\hbar)) / \calF^k_\hbar \calE^\infty (X)((\hbar))) 
  \cong \calE^\infty (X) ,
\]
the $E^1$-term has to coincide with the Hochschild homology of $\calE^\infty (X)$, hence
\begin{equation}
  \label{eq:E1term}
  E^1_{pq} = \Omega^q_{\calE^\infty} (X) .
\end{equation}
Since $\calE^\infty (X)$ is the quotient of $\calC^\infty (M)$ by the ideal 
$\calJ^\infty (X;M)$, it follows from \cite[Sec.~3]{BryDCPM}  that the differential 
$d^1_{pq} : \Omega^q_{\calE^\infty} (X) \rightarrow \Omega^{q-1}_{\calE^\infty} (X)$
coincides with the canonical derivative 
\[
\begin{split}
   \delta : \: & \Omega^q_{\calE^\infty} (X) \rightarrow \Omega^{q-1}_{\calE^\infty} (X), \quad
     f_0 \, df_1 \wedge \ldots \wedge df_q \mapsto \\
     & \sum_{i=1}^q (-1)^{i+1} \{ f_o,f_i \}
     df_1 \wedge \ldots \wedge \widehat{d f_i} \wedge \ldots \wedge df_q +  \\
     & \sum_{1 \leq i < j \leq q} (-1)^{i+j}  f_o d \{ f_i , f_j \} \wedge 
     df_1 \wedge \ldots \wedge \widehat{d f_i} \wedge \ldots \wedge 
     \widehat{d f_j} \wedge \ldots \wedge df_q .
\end{split}
\]
Next let us recall Brylinski's definition of the symplectic Hodge $*$-operator 
(see \cite{BryDCPM}). 
Let $\nu$ be the volume form $\frac{1}{m!} \omega^m$ over $M$, and 
$\Lambda^k\Pi $ the operator 
\[
\begin{split}
  & \Omega^k M \times \Omega^k M \rightarrow \calC^\infty (M),\quad \\
  & (f_0df_1 \wedge \ldots \wedge df_k , g_0dg_1  \wedge \ldots \wedge dg_k) \mapsto
   f_0 g_0 \, \big( \Pi  \lrcorner  \, df_1 \wedge dg_1 \big) \cdot \ldots \cdot  
   \big( \Pi \lrcorner \, df_k \wedge dg_k \big)  .
\end{split}
\]
The symplectic $*$-operator 
$* : \Omega^k (M) \rightarrow \Omega^{2m-k} (M)$  now is uniquely defined by 
requiring that $\alpha \wedge (* \beta) = \Lambda^k\Pi (\alpha , \beta) \, \nu$ for all
$\alpha,\beta \in \Omega^k (M)$. Obviously, $*$ leaves the 
$\calJ^\infty (X;M) \cdot \Omega^\bullet (M)$ invariant, hence induces an operator 
$* : \Omega^k_{\calE^\infty} (M) \rightarrow \Omega^{2m-k}_{\calE^\infty} (M)$ which by the  
properties of the corresponding operator on $\Omega^\bullet (M)$ satisfies the 
equality $* \circ * = \id $. By \cite{BryDCPM} it also follows that on $\Omega^k_{\calE^\infty} (X)$ 
the canonical differential $\delta$ is equal to $( - 1 )^{k+1} * d * $. But this implies 
by \cite{BraPfl} that
\begin{equation}
  \label{eq:E2term}
  E^2_{pq} = H^{2m-q} (X) .
\end{equation}
Under the assumption that $X$ is compact subanalytic, there exists a finite triangulation of $X$, 
hence the singular cohomology with values in $\R$, and by \cite{BraPfl} the 
periodic cyclic homology of $\calE^\infty (X)$ then have to be finite dimensional. 
Arguing like in \cite{NesTsyAIT}, one concludes that under this assumption on $X$, 
the spectral sequence degenerates at $E^2$, and the Hochschild homology of the deformed 
algebra $\calE^\infty (X)[[\hbar]]$ is given by \eqref{eq:E2term}. Let us show that this holds 
even in more generality.

For this more refined computation of the Hochschild and cyclic homology of $\calE^\infty(X)$, we use a specific quasi-isomorphism implementing the isomorphism \eqref{eq:HHDefAlgSmooth} above. In \cite{PPT2ADV}, we have constructed morphisms
\[
\Psi^i_{2k}:\calC_{2k-i}\left(\calC^\infty(M)((\hbar)),\star\right)\to\Omega^i(M)((\hbar)), 
\]
satisfying the property
\begin{equation}\label{eq:psi}
(-1)^id\circ \Psi^i_{2k}=\Psi^{i+1}_{2k}\circ b+\Psi^{i+1}_{2k+2}\circ B,
\end{equation}
where $b$ and $B$ are the Hochschild and Connes' $B$-operator computing cyclic homology. 
>From \cite[Thm 2.4]{PPTLMP}, it follows by Eq.~(\ref{eq:psi}) that the combination $\Psi_i:=\sum_{l\geq 0}\Psi^{2m-2l-i}_{2m-2l}$ defines an S-morphism $\Psi_\bullet $ of complexes of sheaves
\[
\Psi_\bullet:{\rm Tot}_\bullet\left(\mathcal{BC}\left(\calC^\infty(M)((\hbar)),\star\right),b+B\right)\to \left(\bigoplus_{l\geq 0}\Omega^{2m-2l-\bullet}(M)((\hbar)), (-1)^{2m-2l-\bullet}d\right),
\]
where on the left we have the total sheaf complex of Connes' $(b,B)$-complex
(cf. \cite[Prop. 2.5.15]{Loday} for more on S-morphisms).

\begin{proposition} \label{prop:psi}
$\Psi^i_{2k}$ maps $C_{2k-i}\left(\calJ^\infty(X,M)((\hbar))\right)$ to 
$\Omega^i_{\calJ^\infty}(M)((\hbar))$.
\end{proposition}
\begin{proof}
The proof is given by two observations: first, since the Fedosov--Taylor series defining the 
quantization map $\frq:\calC^\infty(M)\to \calW(M)$ only involves partial derivatives, it will 
map $\calJ^\infty(X,M)$ to $\calJ^\infty(X,M)\cdot\calW(M)$. Second, we see from \cite{PPT2ADV} 
that 
$\Psi^i_{2k}$ is given by contraction of an explicitly given cyclic cocycle on the formal Weyl 
algebra acting fiberwise on $\fWeyl(M)$, with the Fedosov connection $D$. From this, the result 
is obvious.
\end{proof}

Proposition \ref{prop:psi} proves that the S-morphism $\Psi_\bullet$ descends to define an 
S-morphism of complexes of sheaves on $X$
\[
\Psi_\bullet:{\rm Tot}_\bullet\left(\mathcal{BC}\left(\calE^\infty(X)((\hbar)),\star\right),b+B\right)
\to \left(\bigoplus_{l\geq 0}\Omega^{2m-2l-\bullet}_{\calE^\infty}(X)((\hbar)), (-1)^{2m-2l-\bullet}d\right).
\]
\begin{theorem}\label{thm:homology} 
Let $(M,\omega)$ be a real analytic symplectic manifold, and $X \subset M$ a subanalytic subset. 
Then the S-morphism $\Psi_\bullet$ defined above is a quasi-isomorphism, and therefore 
\begin{eqnarray*}
&HH_\bullet(\calE^\infty(X)((\hbar)), \star)&=H^{2m-\bullet}(X)((\hbar)),\\
&HC_\bullet(\calE^\infty(X)((\hbar)),\star)&=\bigoplus_{k\geq 0}H^{2m-\bullet-2k}(X)((\hbar)).
\end{eqnarray*}
\end{theorem}
\begin{proof}
The proof is essentially a repetition of the arguments \cite[Theorem 3.9]{PPT2ADV}. Since $\Psi$ 
is an S-morphism, it suffices to check that 
$\Psi^i_{2m}: C_{i}(\calE^\infty(X)((\hbar)))\to \Omega^{2m-i}_{\calE^\infty}((\hbar))$ is a 
quasi-isomorphism. Since $\Psi^i_{2m}$ is a morphism of complexes of sheaves, we only need to 
check that $\Psi^i_{2m}$ is a quasi-isomorphism on a sufficiently nice local chart of $X$, 
which we choose to be the  intersection of a Darboux chart $U$ of $M$ with $X$. 

We note that $\Psi^i_{2m}$ is compatible with the $\hbar$-filtrations on the Hochschild complexes 
$C_{i}(\calE^\infty(X)((\hbar)))$ and $\Omega^{2m-i}_{\calE^\infty}((\hbar))$, and therefore induces a natural morphism between the spectral sequences associated to the $\hbar$-filtrations. To prove that $\Psi^\bullet_{2m}$ is a quasi-isomorphism, it suffices to check that $\Psi^\bullet_{2m}$ is a quasi-isomorphism at the $E^2$-level of the spectral sequences associated to the $\hbar$-filtrations. Over $U$, the algebra $(\calC^\infty(U)((\hbar)), \star )$ can be identified with the standard Weyl algebra. 
In addition, the $E^2$-level of the spectral sequence associated to the Hochschild complex of 
$(\calC^\infty(U)((\hbar)), \star) $ is the Poisson homology complex $(\Omega^\bullet(U)((\hbar)),\delta)$. Similarly, the $E^2$-level associated to $(\calE^\infty(X)((\hbar)),\star )$ is again the Poisson homology complex $(\Omega^\bullet_{\calE^\infty}(X)((\hbar)), \delta)$. Under this identification,  $\Psi^i_{2m}$ becomes the symplectic Hodge star operator, which is an isomorphism between the Poisson homology and the de Rham cohomology in (\ref{eq:E2term}).
\end{proof}
\begin{remark}
Theorem \ref{thm:homology} has a natural generalization to deformation quantizations of 
global Whitney--Poisson structures on $X$ using the method in \cite{DolgADV}, i.e.
\begin{eqnarray*}
&HH_\bullet(\calE^\infty(X)((\hbar)), \star)&=H_\bullet^\pi(X)((\hbar)),\\
&HP_\bullet(\calE^\infty(X)((\hbar)),\star)&=H_\bullet (X)((\hbar)),
\end{eqnarray*}
where $H_\bullet^\pi(X)((\hbar))$ is the Poisson homology of $(X, \pi)$. We leave the details to diligent readers. 
\end{remark}

\begin{remark}
It is easy to see that the so-called ``algebraic index theorem'' \cite{NesTsyAIT} descends to the level of Whitney functions: consider the morphism 
\[
  \mu: C_\bullet\left(\calE^\infty(X)[[\hbar]],\star \right)\to \Omega^\bullet_{\calE^\infty}(X)
\] 
given by
\[
\mu(f_0\otimes\ldots\otimes f_k):=\left. f_0df_1\wedge\ldots\wedge df_k\right|_{\hbar=0},
\]
where $\calE^\infty(X)[[\hbar]]$ is viewed as an algebra over $\C$. 
This map sends the Hochschild differential $b$ to zero and intertwines $B$ with the Whitney--de Rham operator $d$. The previously defined quasi-isomorphism $\Psi$ naturally extends to define a chain morphism 
\[
 \Psi: \Tot_\bullet\left(\mathcal{BC}(\calE^\infty(X)[[\hbar]],\star)\right)\longrightarrow \bigoplus_{l\geq 0}\Omega^{2m-2l-\bullet}_{\calE^\infty}(X)((\hbar)).
\]
The algebraic index theorem gives the defect of the map $\mu$ to agree with the 
morphism $\Psi$:
\begin{theorem}
Under the assumptions of Thm.~\ref{thm:homology} 
the following diagram commutes after taking homology: 
\[
  \xymatrix{\Tot_\bullet\left(\mathcal{BC}(\calE^\infty(X)[[\hbar]],\star)\right)\ar[rr]^{\mu}\ar[drr]_{\Psi}&&
  \bigoplus_{l\geq 0}\Omega^{2m-2l-\bullet}_{\calE^\infty}(X)\ar[d]^{\wedge\hat{A}(M)e^{-\Omega\slash 2\pi\sqrt{-1}\hbar}}
\\
& &
\bigoplus_{l\geq 0}\Omega^{2m-2l-\bullet}_{\calE^\infty}(X)((\hbar))} .
\]
Hereby, $\hat{A}(M)$ is the standard $\hat{A}$-class of $M$ associated to the symplectic structure, and $\Omega$ is the characteristic class of the star product $\star$ on $M$. 
\end{theorem}
As a consequence,  the following equality holds true in $H^\bullet(X)((\hbar))$:
\[
 \Psi(a)= \left([\hat{A}(M)]\cup [e^{\sqrt{-1}\Omega\slash 2\pi\hbar}]\right)\cup \mu(a),
\]
for all $a=a_0\otimes\ldots\otimes a_k\in C_k(\calE^\infty(X),\star)$.
\end{remark}


\bibliographystyle{alpha}

\end{document}